\documentclass{amsart}
\usepackage{amssymb}
\usepackage{amsmath}

\newtheorem{theorem}{Theorem}[section]
\newtheorem{proposition}[theorem]{Proposition}
\newtheorem{observation}[theorem]{Observation}
\newtheorem{corollary}[theorem]{Corollary}
\newtheorem{lemma}[theorem]{Lemma}
\newtheorem{claim}{Claim}[theorem]

\theoremstyle{definition}
\newtheorem{definition}[theorem]{Definition}
\newtheorem{example}[theorem]{Example}

\theoremstyle{remark}
\newtheorem{remark}[theorem]{Remark}

\newcommand{\norm}{{\rm norm}}
\newcommand{\pow}{\mathcal{P}}
\newcommand{\R}{\mathbb{R}}
\newcommand{\bs}{\backslash}
\newcommand{\restrict}{\Rsh}

\newcommand{\dom}{{\rm dom}}
\newcommand{\hn}{{\rm hn}}
\newcommand{\HN}{{\rm HN}}
\newcommand{\partialfn}{ {^{\underline{N}}2} }
\begin{document}
	
\title{Properties of Norms on Sets}
\date{\today}
\author{Cody Anderson}
\address{Department of Mathematics\\
University of Nebraska at Omaha\\
Omaha, NE 68182-0243, USA}
\thanks{I would like to thank Mr Patrick Kerrigan for his generous
  support through the {\em Kerrigan Minigrant Program}.} 
\email{codyanderson@unomaha.edu}

\begin{abstract}
We investigate the norms appearing in the forcing from combinatorial
point of view. We make first steps towards building a catalog of the norms
appearing in multiple settings and sources, reviewing four norms from
Bartoszy\'nski and Judah \cite{Baju95}, Ros{\l}anowski and Shelah
\cite{RoSh:470,   RoSh:628, RoSh:672}, and Shelah \cite{Sh:f}.
\end{abstract}

\maketitle

\section{Introduction}
For many decades, the field of Set Theory couldn't answer one question, ``is
there an infinite set with a size between the natural numbers and the real
numbers?'' In 1940, Kurt G\"odel showed that Set Theory (ZFC) cannot prove
there is a set in between them, but he did not prove there is no set between
them. By 1963, Paul Cohen proved that ZFC was not strong enough to answer
this question. Mathematicians have since used Cohen's ideas in various ways to
prove that similar questions cannot be answered by standard set theory. 
In particular in ``Set Theory of the Reals'' many independence arguments were
given using some sort of measures on finite sets to construct more
complicated objects, forcing notions. This technique of {\em norms on
	possibilities\/} has a relatively long history already. It was first
applied by Shelah \cite{Sh:207} and then used many times by set
theorists. However, they appeared to loose their interest in the norms they
studied right after getting their consistency results. 

The norms might be of interest without the forcing motivation 
behind them as they are, after all, means to measure finite sets. Our goal
is to take first steps towards cataloging, comparing and contrasting many
norms appearing in the theory of forcing for the  reals. 

In this article we investigate four norms, exploring how they work,
identifying properties that show up and constructing proofs of those
properties. We try to compare the various norms with each other.

\subsection{Notation}
Our notation is pretty standard and compatible with that of Bartoszy\'nski
and Judah \cite{Baju95}. In particular, every natural number (a non-negative
integer) is identified with the set of its predecessors. Thus 
\[N=\{0,1,2,\ldots, N-1\}.\]
The set of all natural numbers is denoted $\omega$, so 
\[\omega=\{0,1,2,3,4,5,\ldots\}.\] 
A function $f$ is the set of ordered pairs such that for all $x,y,$ and $z$,
if $(x,y) \in f$ and $(x,z) \in f$ then $y = z$. 

For any sets $A, B$, we define the set $A$ restricted to $B$, 
$$A\restrict B = \{ a\in A : a \subseteq B \}.$$

\subsection{Preliminaries}
\begin{definition}
	A function $\norm: \pow(A) \rightarrow \R$ is {\em a norm on the set  $A$}
	if   
	\begin{enumerate}
		\item $B\subseteq C \subseteq A$ implies $\norm(B) \leq\norm(C)$, and  
		\item if $|A|>1$, then $\norm(A) > 0$ and if $a\in A$ then $\norm(\{ a \}) \leq 1$.  
	\end{enumerate}
\end{definition}

\begin{example}
	[Standard Counting Norm]
	For any set $B$, we define 
	\[ \norm_0(B) = |B| \]
	that is $\norm_0(B)$ is the number of elements in the set $B$. Then, 
	if $A$ is a set with at least 2 elements then $\norm_0\restriction \pow(A)$
	is a norm on $A$ (we will denote it by just $\norm_0$).   
\end{example}

\section{Exclusion Norm}

\begin{definition}
	\label{def1}
	Let $0<F<G$ be natural numbers. For a set $A\subseteq G$ we define
	\[ \norm_1^{F,G}(A) = \frac{F}{|G \backslash A|+1}\]
	If $F,G$ are understood from the context, we may jus write
	$\norm_1(A)$. 
\end{definition}

\begin{observation}
	$\norm^{F,G}_1$ is a norm on $G$.
\end{observation}

\begin{observation}
	[Limiting Cases]
	Let $0<F<G$ be natural numbers.
	\begin{enumerate}
		\item $\norm^{F,G}_1(\emptyset) = \frac{F}{G + 1}$.
		\item $\norm_1^{F,G}(G) = F$.
	\end{enumerate}
\end{observation}

\begin{observation}
	[Relationship with the standard norm] 
	If $\norm_1(A) = k$, then it follows that 
	\[ |A| = G + 1 - \frac{F}{k} \]
	Notice that this suggests a fairly trivial relationship with the standard
	norm, and as such we expect that properties of the standard norm translate (in an altered form) into $\norm_1$. 
\end{observation}

\begin{proposition}
	[Splitting a set]
	\label{prospl}
	Let $F<G$. Suppose we have two sets $A$ and $B$ that partition $G$. Then
	$\norm_1$ of one set is greater than or equal to $\frac{2 F}{G + 2}$ and the
	other is smaller than or equal to this number. 
\end{proposition}

\begin{proof}
	We know that $|A| + |B| = G$ and, without loss of generality, we can
	assume $|A| \geq \frac{G}{2}$ and  $ |B| \leq \frac{G}{2}$. From
	this, it follows that  
	\[ \norm_1(A) = \frac{F}{|G\bs A| + 1} = \frac{F}{|B| + 1} \geq\frac{2 F}{G
		+ 2} \]  
	\[ \norm_1(B) = \frac{F}{|G\bs B| + 1} = \frac{F}{|A| + 1} \leq \frac{2 F}{G
		+ 2} \] 
\end{proof}

Notice how this follows readily from the relationship with the standard
norm. 

\begin{observation}
	We do not have a triangle inequality, that is
	\[\norm_1(A\cup B) \leq \norm_1(A)+\norm_1(B)\]
	does not always hold. Consider $A$ and $B$ that partition $G$, we get 
	\[F = \norm_1(G) \leq \norm_1(A)+\norm_1(B)=\frac{F\cdot
		(|A|+|B|+2)}{(|A|+1) \cdot (|B|+1)}.\]
	When $|A|\cdot |B|\geq 2$, this is a contradiction.
\end{observation}

\begin{observation}
	If $\norm_2(A\cup B) = j$, $\norm_2(A) = k$, and $\norm_2(B) = l$, then we instead have $j \geq \frac{F}{\frac{F}{k} + \frac{F}{l} - G - 1}$, which follows directly from using $|A\cup B| \leq |A| + |B|$ and $|A| = G + 1 - \frac{F}{k}$.
\end{observation}

\section{Interlude: Properties of the choose function}

\begin{lemma}
	\label{inlem1}
	For every positive integers $k,a,b$ with $k-1\leq b\leq a-k$ we have 
	\[\sum\limits_{i=1}^k (-1)^{i-1} {k-1 \choose i-1}{a-i \choose b} = {a-k
		\choose b-k+1}\]
\end{lemma}

\begin{proof}
	We prove the lemma by induction on $k$. 
	
	\noindent Base case $k = 1$: 
	\[\sum\limits_{i=1}^1 (-1)^{i-1}{k-1\choose i-1}{a-i\choose b} =
	(-1)^0 {0\choose 0}{a-1 \choose b} = {a-1 \choose b} = {a-k \choose
		b-k+1}\] 
	Inductive case: Assume the identity is true for $k$ and all suitable $a,b$,
	we will show it for $k+1$. Starting with the formula for $k+1$, pull out the
	$k+1$ term from the sum:
	\[\begin{array}{rl}
	\displaystyle s = & \displaystyle\sum\limits_{i=1}^{k+1} (-1)^{i-1} {k
		\choose i-1} {a - i  \choose b}\\
	\ \\
	\displaystyle s = & \displaystyle (-1)^k{k\choose k}{a - k - 1 \choose b} + 
	\sum\limits_{i=1}^k (-1)^{i-1} {k \choose i-1} {a - i \choose b}
	\end{array}
	\]  
	Applying the identity ${c \choose d} + {c \choose d+1} = {c+1 \choose d+1}$, 
	we may split the sum into two:
	\[s = (-1)^k {a-k-1\choose b} + \sum\limits_{i=1}^k (-1)^{i-1}
	{k-1\choose i-1}{a-i \choose b} + \sum\limits_{i=2}^k (-1)^{i-1}{k-1
		\choose i-2} {a-i\choose b} \] 
	Using the inductive hypothesis we get
	\[\begin{array}{rl}
	\displaystyle s = & \displaystyle (-1)^k {a-k-1\choose b} +
	\sum\limits_{i=2}^k (-1)^{i-1}{k-1 \choose i-2}
	{a-i\choose b} + {a-k \choose b - k + 1}\\
	\displaystyle s = & \displaystyle\sum\limits_{i=2}^{k+1} (-1)^{i-1} {k-1\choose
		i-2}{a-i\choose b} + {a-k\choose b-k+1} 
	\end{array}\]
	Adjusting the sum to run from $i = 1$ to $k$ we get 
	\[s =-\sum\limits_{i=1}^{k} (-1)^{i-1} {k-1\choose i-1}{a-i-1\choose
		b}+{a-k \choose b-k+1}. \] 
	Now using the inductive hypothesis for $k$ and $a-1, b$ we conclude 
	\[s= {a-k\choose b-k+1} - {a-k-1\choose b-k+1} = {a-k-1 \choose b-k} =
	{a -(k+1)\choose b -(k+1) + 1}. \] 
\end{proof}

\begin{lemma}
	\label{inlem2}
	For positive integers $k,a,b$ such that $k-1\leq b$ and $b+k\leq a$ we have  
	\[\sum\limits_{i=1}^k (-1)^{i-1} {k \choose i} {a - i \choose b} =
	\sum\limits_{i=i}^{k} {a - i \choose b - i + 1}.\]
\end{lemma}

\begin{proof} 
	We show the lemma by induction on $k$ (for all $a,b$). 
	
	\noindent Base case $k = 1$:
	\[ \sum\limits_{i=1}^{1} (-1)^{i-1} {1 \choose i}{a -i \choose b} = {a
		- 1 \choose b} = {a - k \choose b-k+1} = \sum\limits_{i=1}^1 {a -
		i \choose b-i+1}  \] 
	Inductive case: Assuming the identity holds for $k$ and all suitable $a,b$
	we  show it holds for $k+1$. Starting with the formula for $k+1$, pull out
	the $k+1$ from the sum: 
	\[\begin{array}{rl}
	\displaystyle s = & \displaystyle\sum\limits_{i=1}^{k+1} (-1)^{i-1} {k+1
		\choose i} {a - i \choose b}\\
	\displaystyle s = &\displaystyle (-1)^{k} {k+1\choose k+1} {a-k-1\choose b} + 
	\sum\limits_{i=1}^{k}(-1)^{i-1}{k+1\choose i}{a-i\choose b}.
	\end{array} \] 
	Using ${c \choose d} + {c \choose d+1} = {c+1 \choose d+1}$, we split the
	sum into two, getting  
	\[s= (-1)^{k} {a-k-1\choose b} +
	\sum\limits_{i=1}^{k}(-1)^{i-1}{k\choose i-1}{a-i\choose b} +
	\sum\limits_{i=1}^{k}(-1)^{i-1}{k\choose i}{a-i\choose b}. \] 
	It follows from the inductive hypothesis for $k$ and $a,b$ that 
	\[s= (-1)^{k} {a-k-1\choose b} +
	\sum\limits_{i=1}^{k}(-1)^{i-1}{k\choose i-1}{a-i\choose b} +
	\sum\limits_{i=1}^k{a - i \choose b-i+1}. \] 
	We can put the $k+1$ back into the first sum getting 
	\[s= \sum\limits_{i=1}^{k+1}(-1)^{i-1}{k\choose i-1}{a-i\choose b} +
	\sum\limits_{i=1}^k{a - i \choose b-i+1}. \] 
	Applying Lemma \ref{inlem2} we conclude 
	\[s= {a-k-1 \choose b - k} + \sum\limits_{i=1}^k{a - i \choose b-i+1}
	= \sum\limits_{i=1}^{k+1}{a - i \choose b-i+1}. \] 
\end{proof}

\begin{lemma}
	\[ \sum\limits_{i=0}^N {N\choose i} n^i = (n+1)^N \]
\end{lemma}

\begin{proof}
	Consider the set ${^{\underline{N}}n}$ of all partial functions from $N$ to
	$n$, where for $f \in {^{\underline{N}}n}$ we have $\dom(f) \subseteq N$
	and $f: \dom(f) \rightarrow n$. Also consider the set ${^{N}(n+1)}$ of
	total functions from $N$ to $n+1$.
	\begin{claim}
		\label{cl1}
		$|{^{\underline{N}}n}| = |{^{N}(n+1)}|$
	\end{claim}
	\begin{proof}[Proof of the Claim]
		Consider the bijection $B:{^N (n+1)}\longrightarrow{^{\underline{N}}n}$ where
		for $f \in {^N (n+1)}$, $B(f) = \{ (a,b) \in f : b \in n \}$.
	\end{proof}
	
	Using Claim \ref{cl1}, we have
	$|{^{\underline{N}}n}| = |{^{N}(n+1)}| = (n+1)^N$. If we have a set
	$z \subseteq N$ such that $|z| = i$, there are $n^i$ functions from $z$ to
	$n$, furthermore, then there are ${N \choose i}$ sets such that $|z| =
	i$.
	If we consider the set $P_i$ of all partial functions $f$ where $|f| = i$,
	we find $|P_i| = {N\choose i} n^i$. Taking the union over all $i$, we get
	${^{\underline{N}}n} = \bigcup\limits_{0\leq i \leq N} P_i$, and as
	$P_i \cap P_j$ is empty when $i \neq j$, we get
	$|{^{\underline{N}}n}| = \sum\limits_{i = 0}^N {N \choose i} n^i = (n+1)^N$.
\end{proof}

\section{Subset Norm}
In this section we will look at a norm introduced (in a different
formulation) in Fremlin and Shelah \cite{FrSh:406}. It was also used in
Bartoszynski and Judah \cite[Section 7.2]{Baju95}. 

\subsection{The norm and its basic properties}

\begin{definition}
	\label{def2}
	Let $n,G$ be positive natural numbers such that $2^n$ divides $G$. 
	\begin{enumerate}
		\item We define the set 
		\[X_n^G = \big\{A\subseteq G:|A|=\frac{G}{2^n}\big\}.\]
		That is to say, $X_n^G$ is the set of all subsets of $G$ that are of
		size $\frac{G}{2^n}$. 
		\item Let $A \subseteq X_n^G$, we define 
		\[\norm_2^{n,G}(A) = \min\{|x| : x \subseteq G \land (\forall a \in A)
		x \nsubseteq a\}.\] 
	\end{enumerate}
	If $n,G$ are understood we may write $\norm_2$ or $X$.
\end{definition}

Until we say otherwise, let $n,G$ be fixed with $2^n|G$ and let
$H=\frac{G}{2^n}$, $\norm_2$ be $\norm_2^{n,G}$ and $X$ be $X_n^G$.  

\begin{lemma}
	\[ \norm_2(X) = H+1 \]
\end{lemma}

\begin{proof}
	Let $x$ be any subset of $G$ such that $|x| = H+1$. By definition, $x
	\not\subseteq a$ for any $a\in X$, and therefore $\norm_2(X) \leq
	H+1$. But if $b$ is any subset of $G$ such that $|b| \leq H$, there is
	an $a\in X$ such that $b \subseteq a$. This is because $X$ contains
	every subset of size $H$ and any subset of size less than $H$ is a
	subset of some set with size $H$. As $X$ contains every set of size
	$H$, it must be that $\norm_2(X) > H$. And thus we have $\norm_2(X) =
	H+1$.  
\end{proof}

\begin{lemma}
	For $A, B \subseteq X$,
	\[ \norm_2(A) + \norm_2(B) \ge \norm_2(A \cup B ) \ge \max(\norm_2(A),
	\norm_2(B)) \] 
\end{lemma}
\begin{proof}
	Suppose we have 
	\[ k = \norm_2(A\cup B) = \min\big\{ |x| : x \subseteq G \land
	\big(\forall a \in (A \cup B)\big)\big(x \nsubseteq a
	\big)\big\}. \]  
	Define: 
	\[\begin{array}{ll}
	Z_A = \big\{ x : x \subseteq G \land \big(\forall a \in A\big)\big(x
	\nsubseteq a \big)\big\},\\
	Z_B = \big\{ x : x \subseteq G \land \big(\forall b \in B\big)\big(x
	\nsubseteq b\big)\big \},\\ 
	Z_{AB} =\big\{ x : x \subseteq G \land \big(\forall a \in (A \cup
	B)\big)\big( x \nsubseteq a\big)\big\}.
	\end{array}\]
	Clearly, $Z_{AB} \subseteq Z_A$ and $Z_{AB} \subseteq Z_B$. By
	definition, we know  
	\[ \begin{array}{l}
	\norm_2(A) = \min\{|x|: x\in Z_A\},\\
	\norm_2(B) = \min\{ |x|: x\in Z_B \},\\
	\norm_2(A\cup B) = \min\{ |x|: x\in Z_{AB} \}.
	\end{array}\]
	Let $l \in Z_{AB}$ be such that $|l| = \norm_2(A\cup B)$. And as
	$Z_{AB} \subseteq Z_A$, it follows that $l \in Z_A$. And consequently,
	$\norm_2(A) \leq |l| = k$. Similarly, $\norm_2(B) \leq
	k$. Consequently, $\norm_2(A\cup B ) \ge \max(\norm_2(A),
	\norm_2(B))$. 
	
	Let $j \in Z_A$ be such that $|j| = \norm_2(A)$ and $l \in Z_B$ such
	that $|l| = \norm_2(B)$. By definition, we know $(\forall a\in A)(
	j\nsubseteq a)$ and similarly, $(\forall b \in B)( l\nsubseteq b)$. From
	these, we know $(\forall a \in A) (j\cup l \nsubseteq a)$ and $(\forall
	b \in B )( j\cup l \nsubseteq b)$. Combining these, we see $j\cup l \in
	Z_{AB}$. Thus we have $\norm_2(A)+\norm_2(B) = |j|+|l| \ge |j\cup l|
	\ge \norm_2(A \cup B) $. 
\end{proof}

It is common (but not universal) for norms to exhibit a triangle inequality. That is to say for a norm, $N$, with sets $A,B$ where $N(A)$ and $N(B)$ are defined, we often have  $N(A\cup B) \leq N(A) + N(B)$. We already know this is true for the standard counting norm, and now we have seen it is true for the subset norm. 

\begin{definition}
	For $A\subseteq X$ and $l\in G$ we define
	\[A(l) = \{x \in A : l\in x\}.\] 
\end{definition}

\begin{lemma}
	For $A\subseteq X$ and $l\in G$ we have 
	\[\norm_2(A(l)) \ge \norm_2(A) - 1. \]
\end{lemma}

\begin{proof}
	Let $A\subseteq X$ and $l_0\in G$.
	
	Let $x\subseteq G$ be such that $|x|=\norm_2\big(A(l)\big)$ and
	$x\not\subseteq a$ for any $a\in A(l)$. Then for each $b\in A$ one of the
	following two cases holds.
	\medskip
	
	\noindent {\sc Case 1}\quad $l\in b$.\\
	Then $b\in A(l)$ so $x\not\subseteq b$ and hence also $x\cup\{l\}
	\not\subseteq b$.
	\medskip
	
	\noindent {\sc Case 2}\quad $l\notin b$.\\
	Then clearly $x\cup\{l\} \not\subseteq b$.
	\medskip
	
	Consequently, the set $x\cup\{l\}$ witnesses that $\norm_2(A)\leq
	\norm_2\big(A(l)\big)+1$. 
\end{proof}

\subsection{Relationship with the standard norm}	
To better understand the behavior of $\norm_2$, we will look at what
the value of $\norm_2(A)$ says about the number of elements in it;
identifying both the smallest number of elements the set $A$ can have
and the largest, while maintaining the value of $\norm_2(A)$.

\begin{proposition}
	[Lower bound when given a value of the norm] 
	\label{lbn2}
	Given an arbitrary non-empty set $A \subseteq X$ such that
	$\norm_2(A)\geq k+1$. Then 
	\[\frac{|A|}{|X|} \ge\frac{(G - H)! (H - k)!}{(G - k)!}.\]
\end{proposition}

\begin{proof}
	There are $ G\choose k$ subsets of $G$ with the size $k$ and each
	subset of size $H$ contains $H\choose k$ of those. For $A$ to contain
	all subsets of size $k$ it must be that $|A| \geq \frac{{G\choose
			k}}{{H\choose k}}$. Comparing this to $|X|$,  
	\[\frac{|A|}{|X|} \ge \frac{{G \choose k}}{{H\choose k}{G \choose H}}
	= \frac{(G - H)! (H - k)!}{(G - k)!}.\]    
\end{proof}

\begin{example}
	Let us consider two extremal cases in Proposition \ref{lbn2}. 
	
	Suppose we have $k = 0$. Then \ref{lbn2} gives us
	$\frac{|A|}{|X|} \ge \frac{(G-H)! H!}{G!}$. This corresponds to
	$\frac{1}{|X|}$, as  expected. If instead we have $k = H$, then \ref{lbn2}
	gives $\frac{(G-H)!(H-H)!}{(G-H)!} = 1$. Again this is expected as the only  
	way to have $\norm_2(A)\geq H+1$ is if $A =X$.  
\end{example}

Using Sterling's approximations for factorials 
\[\sqrt{2\pi}\cdot m^{m+\frac{1}{2}} \cdot e^{-m}\leq m!\leq e\cdot 
m^{m+\frac{1}{2}}\cdot e^{-m},\]
we may give the following conclusion to Proposition \ref{lbn2}.  	

\begin{corollary}
	If $\norm_2(A)\geq k+1$, then 
	\[\begin{array}{r}
	\displaystyle 
	\frac{|A|}{|X|} > \frac{2\pi}{e} \sqrt{\frac{(G-H)(H-k)}{G-k}}
	\frac{(G-H)^G(H-k)^H(G-k)^k}{(G-k)^G (G-H)^H(H-k)^k} =\\
	\ \\
	\displaystyle 
	\frac{2\pi}{e}\sqrt{\frac{(G-H)(H-k)}{G-k}} \frac{(2^n-
		1)^G(1-\frac{k}{H})^H(2^n-\frac{k}{H})^k}{(2^n-\frac{k}{H})^G 
		(2^n -1)^H(1-\frac{k}{H})^k} .
	\end{array}\] 
\end{corollary}

\begin{proposition}
	[Upper bound when given a value of the norm]
	\label{ubn2}
	Given a non-empty set $A\subseteq X$ such that $\norm_2(A)\leq k$, then  
	\[\frac{|A|}{|X|} \leq 1 - \prod\limits_{i = 0}^{k-1} \frac{H - i}{G -
		i}. \]
	Moreover, there is a set $A^*$ with $\norm_2(A^*)=k$ for which the equality
	in the above formula holds.
\end{proposition}

\begin{proof}
	First lets consider the biggest set $A$ such that
	$\norm_2(A)=1$. For any $a \in G$, we let 
	\[A_a \stackrel{\rm def}{=} \{x \in X: a\notin x \}.\]
	This is the biggest a set with a norm of 1 can be, as adding any other
	element from $X$, which must contain $a$, would increase the norm to 2.  It
	is also clear that $|A_a| = {G-1 \choose H}$.
	
	It is worth noticing that for any $B\subset G$ with  $|B| \leq H$, we have  
	\[|\bigcap\limits_{a\in B} A_a| = {G-|B| \choose H}.\] 
	
	If we have $\norm_2(A)\leq k$, then we know for some $B\subset G$ with
	$|B| = k$, $A \subseteq  \bigcup\limits_{a\in B} A_a$. Without loss of
	generality $B=\{1,2,\ldots, k\}$, and we can assume $A\subseteq
	\bigcup\limits_{a=1}^k A_a$, and consequently $|A| \leq
	|\bigcup\limits_{a=1}^k A_a|$.  
	By the general inclusion-exclusion principle, we know 
	\[\begin{array}{r}
	\displaystyle 
	|\bigcup\limits_{a=1}^k A_a| = \sum\limits_{a=1}^k |A_a| -
	\sum\limits_{1\leq a_1 < a_2 \leq k} |A_{a_1} \cap A_{a_2}|
	+\ldots+\\
	\ \\
	\displaystyle 
	(-1)^{i-1} \sum\limits_{1\le a_1 <\ldots<a_i\le k}|A_{a_1}\cap
	\ldots\cap A_{a_i}|+\ldots +(-1)^{k-1}|A_1\cap\ldots\cap A_k|
	\end{array}\]
	But we know $|A_{a_1} \cap A_{a_2} \cap\ldots\cap A_{a_i}| = {G-i
		\choose H}$, so we have 
	\[ |\bigcup\limits_{a=1}^k A_a| = \sum\limits_{i=1}^k(-1)^{i-1}
	{k\choose i}{G-i\choose H} .\] 
	It follows from Lemma \ref{inlem2} that 
	\[ |A|\leq |\bigcup\limits_{a=1}^k A_a| = \sum\limits_{i=1}^k {G-i
		\choose H-i+1} .\] 
	One can easily verify this is equivalent to 
	\[ |A| \leq {G\choose H} - { G-k \choose H-k} \]
	and therefore 
	\[ \frac{|A|}{|X|} \leq  1 - \frac{{G-k \choose H-k}}{{G\choose H}}.\]
	Expanding this using the definition of the choose function, 
	\[ \frac{|A|}{|X|} \leq 1 - \frac{(G-k)!}{G!} \frac{H! (G-H)!}{(H-k)!
		(G - k - (H - k))!} = 1 - \frac{(G - k)!}{G!}\frac{H!}{(H-k)!}\] 
	which then simplifies to
	\[\frac{|A|}{|X|} \leq 1 - \prod\limits_{i = 0}^{k-1} \frac{H - i}{G - i}. \]  
	Concerning the ``Moreover'' part note that
	$\norm_2\big(\bigcup\limits_{a=1}^k A_a\big)=k$. 
\end{proof}

While the above proof demonstrates more of the structure of the norm, there
is a simpler way to arrive at the same result. 

\begin{proof}[Alternative proof of \ref{ubn2}]
	If we have $\norm_2(A) = k$, then we know that there is an $x$ such that
	$|x| = k$ and $(\forall a \in A)(x \nsubseteq a)$. Without loss of
	generality, we will assume that $x = k$. Every set in
	$K = \{ k \cup y : y \subseteq G\bs k \land |y| = H - k\} $ includes
	$k$. We also know that $|K| = {G - k \choose H -k}$. Subtracting these
	sets from $X$, we get
	\[|A| \leq {G \choose H} - {G - k \choose H - k}\]
	Like before, this then becomes
	\[\frac{|A|}{|X|} \leq 1 - \prod\limits_{i = 0}^{k-1} \frac{H - i}{G - i}. \] 
\end{proof}

\begin{remark}
	In \cite[Lemma 7.2.8]{Baju95} the authors claimed that (using our notation)
	\[\frac{|A|}{|X|}\geq 1-\frac{1}{2^{kn}}\quad\mbox{ whenever } \norm_2(A)
	=k+1.\] 
	However, as we showed in Proposition \ref{ubn2} there is a set $A^*$ with
	$\norm_2(A^*)=k$ and $\frac{|A^*|}{|X|}=1-\prod\limits_{i=0}^{k-1}
	\frac{H-i}{G-i}$.  If $i>0$, then $H-i<H-\frac{i}{2^n}$ and hence
	$H-i<(G-i)\cdot \frac{1}{2^n}$. Consequently, for $k>1$ we have
	$\prod\limits_{i=0}^{k-1} \frac{H-i}{G-i}<\frac{1}{2^{nk}}$. Thus
	\cite[Lemma 7.2.8]{Baju95} is not true. However, in the application we may
	use the fact that if $G, H \gg k$, 
	\[ \frac{|A|}{|X|} \leq 1 - \prod\limits_{i = 0}^{k-1} \frac{H - i}{G-
		i} \approx 1 - \prod\limits_{i = 0}^{k-1} \frac{H}{G} \]  
	Using the identity $H = \frac{G}{2^n}$, we get $\frac{|A|}{|X|}$ is
	less than or approximately equal to $1 - 2^{-kn}$.  
\end{remark}

\section{Graph Coloring Norm}

We examine the properties of one of the forcing norms used in Ros{\l}anowski
and Shelah \cite[\S 2.4]{RoSh:470}. In the process of examining this norm, we
found that it bears a strong relationship with the coloring number of
a hyper-graph. We then use this relationship to determine the value of the norm when applied to sets with various properties.

\subsection{Definition and Basis Properties}
Taking $N$ to be any natural number, we define the set $P_N = \{ a \in
\pow(N) : |a|\geq 2 \}$. 

\begin{definition}
	Given an arbitrary set $A\subseteq P_N$, we define the graph coloring norm,
	$\norm_3^N: \pow(P_N) \rightarrow \omega$ as follows: 
	\begin{enumerate}
		\item $\norm_3^N(A) \geq 0$ always
		\item $\norm_3^N(A) \geq 1$ if and only if $A \neq \emptyset$.
		\item $\norm_3^N(A) \geq n+1$ when for every $z \subseteq N$ either
		$\norm_3^N(A\restrict z) \geq n$ or $\norm_3^N(A\restrict (N\bs z)) \geq n$. 
	\end{enumerate}
	(Remember, $A\restrict z = \{ a\in A : a \subseteq z \}$.)
\end{definition}

When $N$ is understood, we may simply abbreviate $\norm_3^N(A)$ as
$\norm_3(A)$. We may also abbreviate $P_N$ as $P$. For the rest of this
section, unless stated otherwise, we will assume we are working with a
fixed $N$ .

\begin{theorem}
	\label{subsets}
	For any sets $A\subseteq B\subseteq P$, the inequality $\norm_3(A) \geq n$
	implies $\norm_3(B) \geq n$.
\end{theorem}

\begin{proof}
	Towards induction, when $n = 0$, this is true by definition.
	
	When $n = 1$ and $A \subseteq B$,
	$$\norm_3(A) \geq 1 \implies A\neq \emptyset \implies B \neq \emptyset \implies
	\norm_3(B) \geq 1.$$ 
	
	Inductive step: Assume we know that for any natural number $k$, and any sets
	$A\subseteq B \subseteq P$, $\norm_3(A) \geq k \implies \norm_3(B) \geq k$. Let
	$A\subseteq B \subseteq P$ be any sets such that $\norm_3(A) \geq k+1$. Take
	$z \subseteq N$, without loss of generality, we can assume
	$\norm_3(A \restrict z) \geq k$. We also know that
	$A\restrict z \subseteq B \restrict z$. By the induction hypothesis, we have
	$\norm_3(B\restrict z) \geq k$. Finally, as $z$ is any subset of $N$, we have
	$\norm_3(B) \geq k+1$.
\end{proof}

Our definition of $\norm_3(A) \geq n$ does not explicitly imply transitivity,
but due to theorem \ref{subsets}, it is easy to see the following. 

\begin{corollary}
	For any set $A\subseteq P$, we have 
	\[\norm_3(A) \geq n+1 \implies \norm_3(A) \geq n.\]
\end{corollary}

\begin{proof}
	This follows from the fact $A\restrict z$ is a subset of $A$, and to have
	$\norm_3(A)\geq n+1$ we must have $\norm_3(A\restrict z) \geq n$ (or
	$\norm_3(A\restrict (N\bs z)) \geq n$).  
\end{proof}

\begin{observation}
	For any set $A\subseteq P$, it is not the case that $\norm_3(A) \geq N+1$.
\end{observation}

\begin{proof}
	To see this, consider what happens as we remove one element from $N$ at a
	time.   
\end{proof}

Note that a better bound will be found later, and we are merely offering
this bound to argue that $\norm_3(A)$ is well--defined. 

\begin{definition}
	For an arbitrary set $A\subseteq P$, we say $\norm_3(A) = n$ if $\norm_3(A) \geq
	n$, but not $\norm_3(A)\geq n+1$. 
\end{definition}

\begin{corollary}
	For any set $A\subseteq P$, $\norm_3(A)$ is well defined.
\end{corollary}

\begin{definition}
	Given the sets $a, p\subseteq N$, and $n\in N$ we say
	\begin{enumerate}
		\item $n$ is a vertex,
		\item $a$ is an edge if $|a| = 2$,
		\item $a$ is a polygon if $|a| \geq 3$,
		\item $a$ is an $n$-gon if $|a| = n$, and
		\item $a$ is an edge of $p$ if $a$ is an edge and $a\subseteq p$.
	\end{enumerate}
\end{definition}

\begin{example} 
	\label{examples}
	\begin{enumerate}
		\item Consider the set $A = \{ \{0,1\} \}$. It is clear that $A$ is
		non-empty, giving us $\norm_3(A) \geq 1$. The set $z = \{ 0 \}$, however,
		gives us $A\restrict z = A \restrict (N\bs z) = \emptyset$. Therefore we
		have $\norm_3(A) = 1$. 
		\item Consider the set $A = \{ \{ 0,1\}, \{ 1,2\}, \{2,0\}  \}$. Not only is
		$A$ non-empty, but for any $z\subseteq N$ either $A\restrict z$ or
		$A\restrict (N\bs z)$ in non-empty. This is because out of $0,1,$ and $2$,
		at least two of them must be in either $z$ or $N\bs z$. If we take $z = \{
		0,1 \}$, we get $A\restrict z = \{ \{0,1\} \}$, and we already know
		$\norm_3(A\restrict z) = 1$. Therefore we have $\norm_3(A) = 2$. 
		\item Consider the set $A = \{ \{0,1\}, \{1,2\}, \{2,3\}, \{3,0\} \}$. The
		set $z = \{ 0,2 \}$ gives $A\restrict z = A\restrict (N\bs z) =
		\emptyset$. This gives us $\norm_3(A) = 1$. Notice that although this set
		contains more element than the set in example 2, we get a smaller norm. 
		\item Consider the set $A = \{ \{0,1,2\} \}$. The set $z = \{ 0 \}$ gives us
		$\norm_3(A) = 1$. Notice that the set of all the edges of this triangle has a
		larger norm that the set consisting of only the triangle. 
	\end{enumerate}
\end{example}

\begin{definition}
	\begin{enumerate}
		\item Given an arbitrary set $A\subseteq P$, and a partition of $N$, $V_0,
		V_1, \ldots, V_{n-1}$, we say {\em $A$ is split by the sets $V_0, \ldots,
			V_{n-1}$\/}   if 
		$$(\forall k \in n)(A\restrict V_k = \emptyset).$$ 
		\item Given an arbitrary set $A\subseteq P$, we say {\em $A$ can be split by $n$
			sets\/}  if there exists a partition of $N$ into $n$ sets $V_0, \ldots,
		V_{n-1}$ such that $A$ is split by $V_0, \ldots, V_{n-1}$.    
	\end{enumerate}
\end{definition}

\begin{theorem}
	Suppose we have an arbitrary set $A\subseteq P$ such that $A$ can be split
	by $2^n$ sets, then it must be that $\norm_3(A) \leq n$. 
\end{theorem}

\begin{proof}
	Towards induction, when $n = 0$, it must that $A = \emptyset$. Therefore, we
	have $\norm_3(A) = 0$. 
	
	When $n=1$, we can find two sets $z$ and $N\bs z$ such that $A \restrict z =
	A \restrict (N\bs z) = \emptyset$, therefore we have $\norm_3(A) \leq 1$. 
	
	Inductive step: Assume we have that for any set $A$, if $A$ can be split by
	$2^k$ sets then $\norm_3(A) \leq k$. If we have a set $A$ that can be split by
	the sets $V_0, V_1,\ldots, V_{2^{k+1}-1}$, take $z= V_0\cup V_1 \cup\ldots\cup
	V_{2^k - 1}$. It must be that $A\restrict z$ can be split by $V_0, V_1, \ldots,
	V_{2^k-2}, (V_{2^k-1} \cup (N\bs z))$. Similarly, it must be that
	$A\restrict (N\bs z)$ can be split by $2^k$ sets. By the inductive
	hypothesis, we have $\norm_3(A\restrict z) \leq k$ and $\norm_3(A\restrict (N\bs
	z)) \leq k$, and consequently, $\norm_3(A)\leq k+1$. 
\end{proof}

\begin{theorem}
	Let $A\subseteq P$. If $A$ cannot be split by $2^n$ sets, then $\norm_3(A) >
	n$. 
\end{theorem}

\begin{proof}
	Towards induction, when $n = 0$, it must be that $A$ is non-empty, therefore
	$\norm_3(A) \geq 1 > 0$. 
	
	When $n = 1$, for any $z\subseteq N$, either $A\restrict z \neq \emptyset$
	or $A\restrict (N\bs z) \neq \emptyset$. Therefore, $\norm_3(A) \geq 2 > 1$. 
	
	Inductive step: Assume that for any set $A'$,  
	\begin{itemize}
		\item if $A'$ cannot be split by $2^k$ sets, then $\norm_3(A') \geq k+1$.
	\end{itemize}
	Suppose now that a set $A$ cannot be split by $2^{k+1}$ sets.
	\begin{claim}
		\label{cl2}
		If $z$ is any subset of $N$, then either $A\restrict z$ cannot be split by
		$2^k$ sets or $A\restrict (N\bs z)$ cannot be split by $2^k$ sets. 
	\end{claim}
	
	\begin{proof}
		Suppose towards contradiction that both the sets $A\restrict z$ and
		$A\restrict (N\bs z)$ can be split by $2^k$ sets. Let $V_0, V_1,\ldots,
		V_{2^k-1}$ be the sets that split $A\restrict z$, and let $V_{2^k}, V_{2^k
			+1},\ldots, V_{2^{k+1}-1}$ be the sets that split $A\restrict (N\bs z)$. Then
		$A$ is split by 
		\[V_0\cap z, V_1\cap z,\ldots, V_{2^k -1}\cap z, V_{2^k}\cap(N\bs z),
		V_{2^k+1}\cap (N\bs z), \ldots, V_{2^{k+1}-1}\cap (N\bs z),\]
		but this contradicts that $A$ cannot be split by $2^{k+1}$ sets. 
	\end{proof}
	
	Suppose $z\subseteq N$. By Claim \ref{cl2}, without loss of generality, we
	may assume that $A\restrict z$ cannot be split by $2^k$ sets. Then by the
	inductive hypothesis, we get $\norm_3(A\restrict z) \geq k +1$. As $z$ can be
	any subset of $N$, we have $\norm_3(A) \geq k+2$.  
\end{proof}

\begin{corollary}
	For any set $A\subseteq P$, if $A$ can be split by $c$ sets, but not by
	$c-1$ sets, and for some $n$ we have $2^n < c \leq 2^{n+1}$, then $\norm_3(A) =
	n+1$. 
\end{corollary}

\begin{corollary}
	If we have an arbitrary set $A\subseteq P_N$, we know that $A\subseteq
	P_{N+1}$, and $\norm_3^N(A) = \norm_3^{N+1}(A)$. 
\end{corollary}

\begin{observation}
	For any $c \leq N$, there exists a set $A$ such that $A$ can be split by $c$
	sets, but not $c-1$ sets. To see this, create the set $C \subseteq N$ such
	that $|C| = c$ and take $A = \{ \{a,b\} : a,b \in C \land a \neq b \}$. 
\end{observation}

\begin{observation}
	Assume $2\leq n,c$ satisfy $c\cdot (n-1)\leq N$.  Then there is a set of
	$n$-gons that can be split by $c$ sets, but not by $c-1$ sets. To see this,
	consider the set $C = c(n-1) \subseteq N$, and the set $A$ of every $n$-gon
	in $C$. 
\end{observation}

\begin{theorem}
	For any sets $A, B \subseteq P$, if $A$ can be split by $m$ sets and $B$ can
	be split by $n$ sets, then $A\cup B$ can be split by $mn$ sets. 
\end{theorem}

\begin{proof}
	Let set $A$ be split by a partition $V_0, V_1,\ldots, V_{m-1}$ and let
	$W_0, W_1,\ldots, W_{n-1}$ be a partition of $N$ splitting the set $B$. 
	Then the family of sets 
	\[\{ V_i\cap W_j: i\in m \land j \in n \}\]
	splits $A\cup B$.    
\end{proof}

\begin{corollary}
	If we have two sets $A, B \subseteq P$, then 
	$$\norm_3(A\cup B) \leq \norm_3(A) + \norm_3(B).$$ 
\end{corollary}

\begin{observation}
	Without additional constraints, this is the best bound we can find.
\end{observation}

\begin{proof}
	Let $n$ be any natural number, and $N = n^2$. We define the family of $n$
	sets $\{ V_0, V_1,\ldots,V_{n-1} \}$ such that $V_k = \{ kn, kn+1,\ldots,
	(k+1)n-1 \}$. We also define the sets   
	\[A = \{ \{a,b\}: a,b\in N \land (\forall k \in n)(a\in V_k\implies b\notin V_k) \}\]
	and 
	\[B = \{ \{a,b\}: a,b \in N \land (\forall k\in n)(a\in V_k \implies b \in
	V_k) \land a \neq b \}.\] 
	Notice that $A$ is split by the sets $V_k$ for $k \in n$, and $B$ is split
	by the sets $U_k = \{ m \in N : m \text{ is congruent to } k \text{ modulo }
	n\}$, and $A\cup B$ can be split by $n^2$ sets. We will argue that $A\cup B$
	cannot be split by $n^2 -1$ sets. Towards contradiction, suppose we have the
	sets $W_0, W_1, \ldots, W_{n^2 - 2}$. By the pigeon hole principle, we know
	at least one set $W$ must contain two elements $a$ and $b$, but we know $\{
	a,b \} \in A\cup B$, which contradicts that $A\cup B$ is split by $W_0,
	\ldots, W_{n^2-2}$. 
\end{proof}

\subsection{From polygons to graphs}

\begin{definition}
	We define the function $f:\pow(N) \rightarrow \omega$ as follows, let $a$ be
	any subset of $N$ and take $a_0, a_1,\ldots, a_k$ such that $a_0 < a_1 <
	\ldots < a_k$, and $a = \{ a_0, a_1, \ldots, a_k \}$. We then say 
	\[f(a) = a_0 N^0 + a_1 N^1 + \ldots + a_k N^k.\] 
\end{definition}

\begin{definition}
	We say a function $g: P \rightarrow P$ is a polygon-reducing function if
	$g(a) \subseteq a$ and $g(a) = a$ if and only if $a$ is an edge. 
\end{definition}

\begin{definition}
	Let $g$ be any polygon-reducing function, we define the function $\psi_g:
	\pow(P) \rightarrow \pow(P)$ as follows. Let $A$ be any non-empty subset of
	$P$, and let $a\in A$ be such that $f(a) = \max\{ f(a'): a'\in A \}$. We
	define   
	\[ \psi_g(A) = (A\bs \{a\})\cup \{g(a)\} .\]
	We also set $\psi_g(\emptyset)=\emptyset$. 
\end{definition}

\begin{theorem}
	For any polygon-reducing function $g$ and any set $A\subseteq P$, 
	$$\norm_3(A)\leq \norm_3(\psi_g(A)).$$ 
\end{theorem}

\begin{proof}
	By induction on $n$ we will show that for every set $A\subseteq P$,
	
	$\norm_3(A)\geq n$ implies $\norm_3\big(\psi_g(A)\big)\geq n$. 
	
	\noindent When $n = 0$ the implication holds by definition. When $n = 1$ and
	$\norm_3(A)\geq n$, it must be that $A$ is non-empty, and this implies
	$\psi_g(A)$ is non-empty, which gives us $\norm_3(\psi_g(A)) \geq 1$.  
	
	\noindent Inductive step: Assume that for every set $A'$, $\norm_3(A')
	\geq k$ implies $\norm_3(\psi_g(A')) \geq k$. Let $A$ be any set such that
	$\norm_3(A) \geq k+1$. Let $z$ be any subset of $N$, without loss of
	generality, we can assume $\norm_3(A\restrict z) \geq k$. If $a \in A$ is such
	that $f(a) = \max\{ f(a'): a' \in A \}$, then there are two cases: 
	\begin{enumerate}
		\item $a \nsubseteq z$, in this case $A\restrict z \subseteq
		\psi_g(A)\restrict z$, giving us $\norm_3(A\restrict z) \leq
		\norm_3(\psi_g(A)\restrict z)$. 
		\item $a\subseteq z$, in this case $\psi_g(A)\restrict z = \psi_g(A\restrict
		z)$, and by the inductive hypothesis we have $k\leq
		\norm_3(\psi_g(A\restrict z)))$.  
	\end{enumerate}
	In both cases we have $k \leq\norm_3(\psi_g(A)\restrict z)$. As $z$ is
	arbitrary, we have 
	\[\norm_3(A) \geq k+1 \implies \norm_3(\psi_g(A)) \geq k+1.\]  
\end{proof}

\begin{corollary}\label{edge-ineq}
	If we have an arbitrary set $A \subseteq P$, and we create the set $A'$ by
	replacing every polygon in $A$ with an edge of that polygon, then $\norm_3(A)
	\leq \norm_3(A')$. 
\end{corollary}

\begin{example}
	Without extra constraints, we cannot guarantee 
	\[\norm_3(A) =\norm_3(\psi_g(A)).\]
	Consider the set $A = \{ \{0, 1\}, \{1,2\}, \{0,1,2\} \}$, and any
	polygon-reducing function $g$ such that $g(\{0,1,2\}) = \{2,0\}$. We get
	$\psi_g(A) = \{\{0,1\},\{1,2\},\{2,0\}\}$. We know that $\norm_3(A) = 1$ as
	$A \restrict \{1\} = \emptyset=A\restrict N\setminus\{1\}$, and
	$\norm_3(\psi_g(A)) = 2$, as explained in Example \ref{examples}(2).
\end{example}

\begin{definition}
	Let $E$ be the set of all edges of $N$. For a set $A \subseteq
	P$, we define the set 
	\[\mathcal{E}_A = \{ E' \subseteq E: (\forall e \in
	E')(\exists p \in A) e \subseteq p \land (\forall p \in A)(\exists e \in
	E')e\subseteq p \}.\] 
	That is $E'\in {\mathcal E}_A$ if every edge in $E'\subseteq E$ is an edge of
	some element of $A$ and every element of $A$ has an edge in $E'$.  
\end{definition}

\begin{theorem}
	Given an arbitrary set $A\subseteq P$,
	\[ \norm_3(A) = \min\{\norm_3(E'): E' \in \mathcal{E}_A\} .\]
\end{theorem}

\begin{proof}
	From Corollary \ref{edge-ineq}, we already have $\norm_3(A) \leq
	\min\{\norm_3(E'): E' \in \mathcal{E}_A\}$, therefore if we show $\norm_3(A) \geq
	\min\{\norm_3(E'): E' \in \mathcal{E}_A\}$, the proof will be complete. Let $n$
	be such that $A$ can be split by $n$ sets, but not by $n-1$ sets, and let
	$V_0, V_1, \ldots, V_{n-1}$ be a partition of $N$ such that $A$ is split by
	$V_0, V_1, \ldots, V_{n-1}$. For every $a \in A$, fix an edge $e_a = \{ a_0,
	a_1 \}$ where for some $k \in n$ we have $a_0 \in V_k$ and $a_1 \notin
	V_k$. Notice that $\{ e_a : a\in A \}$ is an element of $\mathcal{E}_A$, and
	$\{ e_a : a\in A \}$ is split by $V_0, \ldots, V_{n-1}$. This gives us that
	$\min\{\norm_3(E'): E' \in \mathcal{E}_A\} \leq \norm_3(\{ e_a : a\in A \}) \leq
	\norm_3(A)$. Combining both inequalities, we have $\min\{\norm_3(E'): E' \in
	\mathcal{E}_A\} = \norm_3(A)$. 
\end{proof}	

\begin{observation}
	Suppose we have an arbitrary set $A\subseteq P$ such that for all $a \in A$,
	$a$ is an edge. Then the pair $(N,A)$ can be interpreted as a graph, with
	$N$ being a set of vertexes and $A$ being a set of undirected edges. 
\end{observation}

\begin{definition}
	We say a graph $G = (V, E)$ is $c$--colorable if we can find a partition 
	\[V=V_0\cup V_1\cup \ldots\cup V_{c-1}\]
	such that $(\forall k \in c)(E\restrict V_k= \emptyset)$. 
\end{definition}

\begin{observation}
	Given an arbitrary set $A \subseteq \{ \{a,b\}: a,b \in N \land a\neq b \}$,
	$A$ can be split by $c$ sets if and only if the graph $(N,A)$ is
	$c$--colorable. 
\end{observation}

\begin{definition}
	We say a graph $G$ has chromatic number $c$ if it is $c$-colorable, but not
	$c-1$--colorable. 
\end{definition}

\subsection{The Norm of Specific Sets}

\begin{proposition}
	For any $k \in N$, the set $\{ a \in \pow(N) : a \text{ is a } k\text{-gon}
	\}$ can be split by $\min\{\lceil \frac{N}{k-1} \rceil, \lfloor \frac{N}{k}
	+ 1 \rfloor\}$ sets. 
\end{proposition}

\begin{proof}
	To ensure every $k$-gon is split, we must make sure every set contains fewer
	than $k$ elements in it. By the pigeon hole principle, if there are $n$
	sets, at least one must have at least $\lceil\frac{N}{n}\rceil$ elements. So
	we need $\lceil \frac{N}{n} \rceil < k$, removing the ceiling operator and
	rearranging, we must have $n > \frac{N}{k}$. We can guarantee this with $n =
	\min\{\lceil \frac{N}{k-1} \rceil, \lfloor \frac{N}{k} + 1 \rfloor\}$. 
\end{proof}

\begin{remark}
	As we are working with $2 \leq k \leq N$, this reduces to just $\lfloor \frac{N}{k} + 1 \rfloor$.
\end{remark}

\begin{proposition}
	Given an arbitrary set $A\subseteq P$ and a vertex $v\in N$, if we create a
	set $B = \{ a \in A : v \notin a \}$, then we know $\norm_3(B) \geq \norm_3(A)
	-1$. 
\end{proposition}

\begin{proof}
	The proposition is trivially true for $\norm_3(A)\leq 1$, so suppose
	$\norm_3(A)=k+1$, $k>0$. We know that $\norm_3(A) \geq k+1$ if and only if for
	every $z \subseteq N$ we have either $\norm_3(A\restrict z) \geq k$ or
	$\norm_3(A\restrict (N\bs z)) \geq k$. Taking $z = \{v\}$, we get $A\restrict
	z = \emptyset$, and it must be that $\norm_3(A\restrict (N\bs z)) \geq k$. It
	is also clear that $B = A\restrict (N\bs z)$. 
\end{proof}

\begin{proposition}
	Given a vertex $v\in N$, then $\norm_3\big(\{ a \in P : v \in a \}\big)=1$.   
\end{proposition}

\begin{proof}
	Clearly the set $B=\{ a \in A : v \in a \}$ is not empty. Let $z = \{ v
	\}$. As every element of $B$ contains at least two elements, it must be that
	$B \restrict z = \emptyset$, but as every element of $B$ contains $v$, it
	must be that $B\restrict (N\bs z) = \emptyset$. Consequently, $\norm_3(B)=
	1$.  
\end{proof}

\begin{remark}
	Notice that these propositions show that $\norm_3$ and $\norm_2$ have a poor relationship, that is, if $\norm_3(A)$ and $\norm_2(A)$ are both defined, $\norm_3(A)$ may be large while $\norm_2(A)$ is small, or vise-versa.
\end{remark}

\subsection{Relationship with the Standard Norm}

\begin{theorem}
	Given an arbitrary set $A \subseteq P$ such that $\norm_3(A) = k$, we know that
	$|A| \geq {2^{k-1} + 1 \choose 2}$. 
\end{theorem}

\begin{proof}
	The smallest $A$ can be to achieve a particular value of the graph coloring
	norm is when the elements of $A$ are all edges. This is a consequence of
	corollary \ref{edge-ineq}. Therefore, we will assume every element of $A$ is
	an edge, and it must be that $A$ cannot be split by at most $2^{k-1}$
	sets. The smallest set $A$ that cannot be split by $2^{k-1}$ sets is when
	$(N,A)$ forms a complete graph on $2^{k-1}+1$ vertexes. Therefore we have ${2^{k-1}+1 \choose 2}$ edges in $A$. 
\end{proof}

\begin{theorem}
	Given an arbitrary set $A \subseteq P$ where $\norm_3(A) = k$ it must be that
	$|A| \leq 2^N - 2^k2^{\frac{N}{2^k} }+ 2^k - 1$. 
\end{theorem}

\begin{proof}
	For every $n$ there are ${N \choose n}$ possible $n$-gons. Because $\norm_3(A)
	= k$, we know that some partition $N=V_0\cup V_1\cup\ldots\cup V_{2^{k}-1}$
	splits $A$ . For convenience, we will say $x_0 = |V_0|,\ x_1 = |V_1|,\
	\ldots ,\ x_{2^k-1} = |V_{2^k - 1}|$. As $A$ can be split by $2^k$ sets, we
	have over counted the number of possible $n$-gons, considering this, there
	are ${N\choose n} - \sum\limits_{l = 0}^{2^k -1} {x_l \choose n}$ possible
	$n$--gons in $A$. Summing over all $n$ and keeping the convention that
	${x_l\choose n}=0$ for $n>x_l$, we get $|A| \leq \sum\limits_{n=2}^N
	\Big({N\choose n} - \sum\limits_{l = 0}^{2^k - 1}{x_l  \choose n}\Big)$. One
	can verify that this simplifies to 
	\[|A| \leq 2^N-N-1+ \sum\limits_{l = 0}^{2^k - 1}(1 + x_l - 2^{x_l}),\] 
	and as $N = x_0 + x_1 +\ldots + x_{2^k-1}$, this reduces to $|A| \leq
	2^N+2^k-1-\sum\limits_{l =0}^{2^k-1}2^{x_l}$. 
	
	Interpreting this as a function on real numbers
	\[h(x_0,x_1,\ldots, x_{2^k-1}) = 2^N+2^k-1-\sum\limits_{l =
		0}^{2^k-1}2^{x_l},\] 
	we know from calculus that if this function has a maximum, it occurs when
	the gradient of the function is zero. Taking $x_{2^k - 1} = N - x_0 - x_1 -
	\ldots - x_{2^k - 2}$, and applying the partial derivative to any of the
	other variables, we get $\frac{\partial h}{\partial x_l} = - \ln(2)2^{x_l} +
	\ln(2)2^{N-x_0-x_1-\ldots-x_l-\ldots-x_{2^k-2}}$. The partial derivative is
	zero only when $x_l = N - x_0 - x_1 - \ldots - x_l - \ldots -x_{2^k-2} =
	x_{2^k - 1}$. Furthermore, this holds for all $l < 2^k$, which gives us
	$\textrm{grad}(h) = 0$ when $x_0 = x_1 = \ldots = x_{2^k-1} =
	\frac{N}{2^k}$. To verify that this is the maximum, we test points where for some $m$ we have $x_m = N$ and $x_l = 0$ for $l \neq m$. We get $h(0,0, \ldots, N,\ldots,0) = 0$, and 
	when the gradient is zero, we get $h(\frac{N}{2^k},\ldots,\frac{N}{2^k}) =
	2^N - 2^k2^{\frac{N}{2^k} }+ 2^k - 1$. Consequently, we have $|A| \leq 2^N -
	2^k2^{\frac{N}{2^k} }+ 2^k - 1$.  
\end{proof}

\section{Hall Norms}
Interesting norms were introduced by Roslanowski and Shelah
\cite{RoSh:628, RoSh:672} to construct ccc forcing notions. They have quite
special properties allowing ``gluing and cutting'', cf. Propositions
\ref{13X} and \ref{13B} below.  

In this section, we will assume we are working with a fixed natural number
$N$.  Also, the symbol ``$\restriction$'' is used here to denote the operation
of the restriction of functions. 

\subsection{Definitions and Basic Properties}

\begin{definition}
	\begin{enumerate}
		\item We define $^{\underline{N}}2$ to be the set of all partial functions
		$\sigma$ where  
		\begin{enumerate}
			\item $\dom(\sigma) \subseteq N$, and 
			\item $\sigma: \dom(\sigma) \longrightarrow \{0,1\}$
		\end{enumerate}
		Note that we do include the empty function in $\partialfn$.
		\item The set of all total functions $f : N \longrightarrow \{0,1\}$ is
		denoted by ${^N2}$.
		\item For any two functions $\sigma, \rho \in \partialfn$ we say {\em $\sigma$
			extends $\rho$} if $\rho \subseteq \sigma$. 
		\item For any function $\sigma \in \partialfn$, we define 
		\[ [\sigma] = \{ f \in {^N2} : f \text{ extends } \sigma \} \]
	\end{enumerate}
\end{definition}

\begin{proposition}
	For any two functions $\sigma_1,\sigma_2 \in \partialfn$, if there is an $n
	\in \dom(\sigma_1)$ such that $n \in \dom(\sigma_2)$ and $\sigma_1(n) \neq
	\sigma_2(n)$, then $[\sigma_1]\cap [\sigma_2] = \emptyset$. Otherwise
	$\sigma_1\cup \sigma_2\in \partialfn$ and 
	$$[\sigma_1]\cap [\sigma_2] = [\sigma_1 \cup \sigma_2].$$
\end{proposition}

\begin{proof}
	By definition, we have
	$$[\sigma_1] = \{ f \in {^N2} : \sigma_1 \subseteq f \} = \{ f \in {^N2} :
	(\forall a \in \sigma_1) a \in f \}.$$ 
	Similarly, we have $[\sigma_2] = \{ f : (\forall a \in \sigma_2) a \in f
	\}$. Combining these, 
	$$[\sigma_1]\cap[\sigma_2] = \{ f: (\forall a \in \sigma_1) a \in f \land
	(\forall a \in \sigma_2) a \in f \}.$$ 
	If for some $n$, we have $n\in \dom(\sigma_1)$, $n \in \dom(\sigma_2)$, and
	$\sigma_1(n) \neq \sigma_2(n)$, then there is no function $f$ that can
	satisfy both $(n,\sigma_1(n))\in f$ and $(n, \sigma_2(n))\in f$. Therefore
	$[\sigma_1]\cap[\sigma_2] = \emptyset$. Otherwise, if there is no such $n$,
	then $$[\sigma_1] \cap [\sigma_2] = \{ f: (\forall a \in (\sigma_1\cup
	\sigma_2)) a \in f \} = [\sigma_1 \cup \sigma_2].$$ 
\end{proof}

\begin{definition}
	For $A \subseteq {^N2}$, we define 
	\[ \Delta_N(A) = \{ \sigma \in \partialfn: [\sigma] \cap A = \emptyset \land
	(\forall \rho \subsetneq \sigma) [\rho]\cap A \neq \emptyset \} \] 
	When $N$ is understood through context, we may simply write $\Delta(A)$ for
	$\Delta_N(A)$.  
\end{definition}

\begin{lemma} \label{subset delta}
	For any set $A \subseteq {^N2}$, if $\sigma \in \partialfn$ is such that
	$[\sigma]\cap A = \emptyset$, then there is a $\rho \subseteq \sigma$ such
	that $\rho\in \Delta(A)$. 
\end{lemma}

\begin{proof}
	First note that $[\emptyset] = {^N2}$ and if $A$ is empty, then $\emptyset
	\in \Delta(A)$, otherwise we go through the following process. 
	\begin{enumerate}
		\item Construct the set $P_1 = \{ a \subseteq \sigma : |a| = 1 \}$. If there
		is an $a \in P_1$ such that $[a] \cap A = \emptyset$, then the proof is
		complete. 
		\item Otherwise, construct the set $P_2 = \{ a \subseteq \sigma : |a| = 2
		\}$. If there is an $a \in P_2$ such that $[a] \cap A = \emptyset$, then
		the proof is complete. 
		\item Otherwise $\ldots$
		\item Otherwise, construct the set $P_{|\sigma|} = \{ a \subseteq \sigma :
		|a| = |\sigma| \}$. We know $\sigma \in P_{|\sigma|}$ and $[\sigma] \cap A
		= \emptyset$, therefore the proof is complete. 
	\end{enumerate}
\end{proof}

\begin{proposition}
	Suppose we have two distinct sets $A, B \subseteq {^N2}$, then we know
	$\Delta(A) \neq \Delta(B)$. 
\end{proposition}

\begin{proof}
	Without loss of generality, we can assume that there is an $f \in B$ such
	that $f\notin A$. As $f \notin A$, we know $[f] \cap A = \emptyset$, and for
	some $\sigma_0 \subseteq f$ we have $\sigma_0 \in \Delta(A)$. However, as $f
	\in B$ we know that for every $\sigma \subseteq f$ we have $f \in ([\sigma] \cap
	B)$ and therefore $\sigma_0 \notin \Delta(B)$. 
\end{proof}

\begin{definition}
	For any two sets $\delta_1, \delta_2 \subseteq \partialfn$, we say $\delta_1
	\preceq \delta_2$ if and only if 
	\[(\forall \sigma \in \delta_1)(\exists \rho \in\delta_2)( \rho \subseteq
	\sigma).\]  
\end{definition}

\begin{observation}
	The relation $\preceq$ is transitive (on subsets of $\partialfn$). 
\end{observation}

\begin{lemma}\label{add fn}
	For any set $A \subseteq {^N2}$ and any function $f \in {^N2}$ we know that
	\[\Delta(A\cup \{f\}) \preceq \Delta(A).\] 
\end{lemma}

\begin{proof}
	Consider the set $\Sigma = \{ \sigma : \sigma \in \Delta(A\cup\{f\}) \land
	\sigma \notin \Delta(A) \}$. If $\Sigma$ is empty, then the proof is
	complete. Otherwise, let $\sigma_0$ be any element of $\Sigma$. As $\sigma_0
	\in \Delta(A \cup \{f\})$, we know that $[\sigma_0] \cap A = \emptyset$, and
	by Lemma \ref{subset delta} there is a $\rho \subseteq \sigma_0$ such that
	$\rho \in \Delta(A)$. As $\sigma_0$ is arbitrary, we know that for every
	$\sigma \in \Delta(A \cup \{ f \})$ there is a $\rho\in \Delta(A)$ such that
	$\rho\subseteq \sigma$. 
\end{proof}

\begin{theorem}
	For any two sets $A, B \subseteq {^N2}$, $A$ is a subset of $B$ if and only
	if $\Delta(B) \preceq \Delta(A)$. 
\end{theorem}

\begin{proof}
	First, let us assume that $A\subseteq B$. By a repeated application of Lemma
	\ref{add fn} we easily show that $\Delta(B) \preceq \Delta(A)$.
	
	Now, suppose that $A \nsubseteq B$, so there is an $f \in A$ such that $f
	\notin B$. From Lemma \ref{subset delta}, we know there is a $\sigma_0
	\subseteq f$ such that $\sigma_0 \in \Delta(B)$. As $f \in A$, we know that
	for every $\rho \subseteq \sigma_0$ we have $f \in [\rho] \cap A$, and $\rho
	\notin \Delta(A)$. Therefore, $\Delta(B) \npreceq \Delta(A)$. 
\end{proof}

\begin{definition}
	For any set $\delta \subseteq \partialfn$, we define
	\[\begin{array}{rl}
	\hn(\delta) = &\max\Big\{ k+1: k \in N \mbox{ and for every  }\delta'
	\subseteq \delta \mbox{ there exists }\delta'' \subseteq \delta'\\
	&\quad \mbox{such that the elements of $\delta''$ have pairwise disjoint
		domains  and}\\
	&\quad|\bigcup\limits_{\sigma \in \delta''} \dom(\sigma)| \geq k |\delta'|
	\Big\},\\
	\HN(\delta) = &\max\{ \hn(\delta') : \delta \preceq \delta' \}.
	\end{array}\] 
\end{definition}

\begin{remark}
	In general, we do not have $\delta_1 \subseteq \delta_2
	\subseteq \partialfn$ implies that $\hn(\delta_1) \leq \hn(\delta_2)$. We
	also do not have that $A \subseteq B \subseteq {^N2}$ implies that
	$\hn(\Delta(A)) \leq \hn(\Delta(B))$. Consider the case where $N = 4$, let
	$f_1 = \{ (0,0), (1,0), (2,0), (3,0) \}$ and $f_2 = \{ (0,1), (1,1), (2,1),
	(3,1) \}$. Taking $A = \{f_1\}$ and $B = \{ f_1, f_2\}$, we have $\hn(A) =
	5$ and $\hn(B) = 3$, but we also have $\hn(\Delta(A)) = 2$ and
	$\hn(\Delta(B)) = 1$. 
\end{remark}

\begin{definition}
	For $\delta \subseteq \partialfn$ and $k$, we say that a function $F:
	\delta \rightarrow [N]^k$ is a $k$-selector if for $\sigma, \rho \in \delta$
	we have 
	\begin{enumerate}
		\item $F(\sigma) \subseteq \dom(\sigma)$, and 
		\item $F(\sigma) \cap F(\rho) = \emptyset$ if and only if $\sigma \neq \rho$. 
	\end{enumerate}
\end{definition}

\begin{remark}
	\label{rem6.15}
	For any set $\delta \subseteq \partialfn$, $\hn(\delta) = k+1>1 $ implies
	the existence of a $k$--selector. To see this, notice that the necessary and 
	sufficient condition for the existence of a $k$-selector is similar to the
	necessary and sufficient condition in Hall's marriage theorem. Whereas in
	$\hn(\delta)$, we are only counting the contribution of certain sets, rather
	than all sets. 
\end{remark}

\begin{proposition}
	\label{n6.16}
	Let $\delta\subseteq\partialfn$, $0<k<N$. Then $\HN(\delta)>k$ if and only
	if there is $\delta^*\in\partialfn$ such that $\delta\preceq\delta^*$ and
	elements of $\delta^*$ have disjoint domains of size $k$
	
	In particular, the existence of a $k$--selector for
	$\delta\subseteq\partialfn$ implies $\HN(\delta)\geq k+1$. 
\end{proposition}

\begin{proof}
	$(\Rightarrow)$\quad If $\HN(\delta)>k$ then we may choose
	$\delta'\succeq\delta$ such that $\hn(\delta')>k$. By Remark \ref{rem6.15},
	there is a $k$--selector $F$ for $\delta'$. Then $\delta^*=\{\sigma
	\restriction  F(\sigma):\sigma\in\delta'\}\succeq \delta'$ is as required.  
	\medskip
	
	\noindent $(\Leftarrow)$\quad If $\delta^*\succeq\delta$ is such that all
	elements of $\delta^*$ have disjoint domains of size $k$, then
	$\HN(\delta)\geq \hn(\delta^*)=k+1$. 
\end{proof}

\begin{definition}
	For any set $A\subseteq {^N2}$, we define $\norm_4^N(A) = \HN(\Delta_N(A))$.
	When $N$ is understood through context, we may simply write $\norm_4(A)$.
\end{definition}

\begin{corollary}
	For any two sets $A \subseteq B \subseteq {^N2}$ we have $\norm_4(A) \leq
	\norm_4(B)$. 
\end{corollary}

\begin{definition}
	For any set $\delta \subseteq \partialfn$ we define the set 
	\[D_N(\delta) =\{ f\in {^N2}: (\forall \sigma \in \delta)( \sigma \nsubseteq
	f) \}.\]  
	As $N$ is understood, we may simply write $D(\delta)$.
\end{definition}

\begin{remark}
	There are two distinct sets $\delta_1,\delta_2 \in \partialfn$ such that
	$D(\delta_1) = D(\delta_2)$. Consider the sets $\delta_1 =\big\{
	\{(0,0)\},\{(0,1)\}\big\}$ and $\delta_2 = \big\{ \{(1,0)\}, \{(1,1)\}
	\big\}$. We know that $D(\delta_1) = D(\delta_2) = \emptyset$. Therefore, we
	cannot in general have $\delta = \Delta(D(\delta))$. 
\end{remark}

\begin{observation}
	$\norm_4$ does not exhibit the triangle inequality. Consider the sets $A = D(\{(0,0)\})$ and $B = D(\{(0,1)\})$. We have $\norm_4(A) = \HN(\{(0,0)\}) = 2$, $\norm_4(B) = 2$, but $A \cup B = {^N2}$, thus $\norm_4(A\cup B) = \HN(\Delta({^N2})) = N+1$.
\end{observation}

\begin{proposition}
	For any set $A \subseteq {^N2}$, we have $A = D(\Delta(A))$.
\end{proposition}

\begin{proof}
	First assume $f \in A$ is arbitrary. Clearly, $(\forall \sigma \in
	\Delta(A))( [\sigma]\cap \{f\} = \emptyset)$, that is $f$ does not extend
	any $\sigma\in\Delta(A)$. Consequently, $f\in D(\Delta(A))$.  
	\medskip
	
	Suppose now that $f \notin A$, so $f\in {^N2}\bs A$.  Then we have $[f] \cap
	A = \emptyset$ and by Lemma \ref{subset delta}, there is a $\rho \subseteq
	f$ such that $\rho \in \Delta(A)$.  This gives us $f \notin D(\Delta(A))$. 
\end{proof}

\begin{definition}
	\begin{enumerate}
		\item For any set $\delta\subseteq \partialfn$ and a set $Z\subseteq N$, we
		define the set $\delta_Z = \{ \sigma \in \delta : \dom(\sigma) \subseteq 
		Z\}$.  
		\item For $\sigma \in \partialfn$ and a set $Z\subseteq N$, we  define the
		restriction $\sigma\restriction Z = \{ (a,b)\in \sigma : a \in Z \}$. 
	\end{enumerate}
\end{definition}

\begin{theorem}
	For any set $\delta\subseteq\partialfn$ and a set $Z\subseteq N$, we have
	$\HN(\delta) \leq \HN(\delta_Z)$. 
\end{theorem}

\begin{proof}
	Let $\delta'$ and $k$ be such that $\delta \preceq \delta'$ and
	$\hn(\delta') = \HN(\delta) = k+1$. We know that $\delta'_Z \subseteq
	\delta'$, and by $\hn(\delta') = k+1$, we know $\hn(\delta'_Z) \geq k+1$. We
	also know that $\delta_Z \preceq \delta'_Z$. Therefore, $\HN(\delta_Z) \geq
	\hn(\delta'_Z) \geq \hn(\delta') = \HN(\delta)$. 
\end{proof}

\begin{remark}
	In general, we do not have equality. Consider functions $f_1,
	f_2, f_3 \in {^42}$, where 
	\begin{itemize}
		\item $f_1 = \{ (0,1),(1,1),(2,1),(3,1) \}$, 
		\item $f_2 = \{(0,1),(1,0),(2,1),(3,1) \}$, 
		\item $f_3 = \{(0,0),(1,0),(2,1),(3,1)\}$. 
	\end{itemize}
	For the set $A = \{f_1, f_2, f_3\}$, we have $\Delta(A) = \{
	\{(0,0),(1,1)\},\{(2,0)\},\{(3,0)\} \}$. Thus $\HN(\Delta(A)) =  
	2$ and we also have $\HN(\Delta_{\{0,1\}}(A)) = 3$.
\end{remark}

\begin{definition}
	For $\delta\subseteq \partialfn$ and a set $Z\subseteq N$, we define the
	sets $L(\delta, Z)$ and $R(\delta, Z)$ as follows: 
	\begin{enumerate}
		\item $L(\delta, Z) = \big\{ \sigma\restriction Z : \sigma \in \delta \land
		|\sigma \restriction Z | \geq |\sigma \restriction (N\bs Z)| \big\}$, 
		\item $R(\delta, Z) = \big\{ \sigma\restriction (N\bs Z) : \sigma \in \delta
		\land |\sigma \restriction Z | < |\sigma \restriction (N\bs Z)| \big\}$.  
	\end{enumerate}
\end{definition}

\begin{observation}
	Suppose that $\delta \subseteq \partialfn$ and $Z\subseteq N$. Then 
	\begin{enumerate}
		\item $\delta \preceq L(\delta, Z) \cup R(\delta, Z)$, and 
		\item if $R(\delta, Z) = \emptyset$ then $\delta \preceq L(\delta, Z)$, and
		similarly if $L(\delta, Z) = \emptyset$ then $\delta \preceq R(\delta,
		Z)$. 
	\end{enumerate}
\end{observation}

\begin{remark}
	If $\delta \subseteq \partialfn$ and $Z\subseteq N$ are such that 
	$R(\delta, Z) \neq \emptyset$, then it is possible to have $\HN(L(\delta,
	Z)) > \HN(\delta)$. Let $N = 4$ and consider the functions $\sigma_0 = \{
	(0,0),(1,0),(2,0)\}$, $\sigma_1 = \{ (3,0) \}$. Let $\delta = \{ \sigma_0,
	\sigma_1 \}$ and let $Z = \{ 0,1,2 \}$. We have $L(\delta, Z) =
	\{\sigma_0\}$ and $R(\delta, Z) = \{ \sigma_1 \}$, and consequently
	$\HN(L(\delta, Z)) = 4$, while $\HN(\delta) = 2.$ 
\end{remark}

\begin{theorem}
	\label{thm6.30}
	For any $\delta \subseteq \partialfn$ and $Z \subseteq N$, 
	$$\HN(L(\delta, Z) \cup R(\delta, Z)) = \min\{\HN(L(\delta, Z)),
	\HN(R(\delta, Z))\}.$$ 
\end{theorem}

\begin{proof}
	Let $\delta \subseteq \partialfn$ and $Z\subseteq N$. Let $L(\delta, Z)
	\preceq L$ and $R(\delta, Z) \preceq R$ be such that $\HN(L(\delta, Z)) = 
	\hn(L)$ and $\HN(R(\delta, Z)) = \hn(R)$ and domains of functions from $L$
	are included in $Z$ and domains of functions from $R$ are disjoint from
	$Z$. Clearly we have $L(\delta, Z) \cup  R(\delta, Z) \preceq L\cup R$, and
	hence   
	\[\HN(L(\delta, Z) \cup R(\delta, Z)) \geq \hn(L\cup R) = \min\{\hn(L),
	\hn(R)\}.\]  
	To show the equality, without loss of generality, assume that $\min\{\hn(L), 
	\hn(R) \} = \hn(L)$. Suppose towards contradiction that there is a set
	$\delta' \subseteq \partialfn$ such that $L(\delta, Z) \cup R(\delta, Z)
	\preceq \delta'$ and $\hn(\delta') > \hn(L)$. Consider the set $\delta'_Z$,
	it must be that $L(\delta, Z) \preceq \delta'_Z$ and therefore
	$\HN(L(\delta, Z)) > \hn(L)$, a contradiction. 
\end{proof}

\begin{proposition}
	\label{13X}
	Let $N<M$. Suppose $\emptyset \neq A_1\subseteq {}^N 2$, $\emptyset \neq A_2
	\subseteq {}^{[N,M)} 2$ and $A\subseteq {}^M 2$ are such that
	$\norm_4^N(A_1)>1$ and $\norm^{[N,M)}_4(A_2)>1$ and 
	\[A=\{f\cup g:f\in A_1\land g\in A_2\}.\]
	Then $\norm_4^M(A)\geq \min\big\{\norm_4^N(A_1),
	\norm^{[N,M)}_4(A_2)\big\}$. 
\end{proposition}

\begin{proof}
	Let $\norm^N_4(A_1)=k+1$ and $\norm^{[N,M)}_4(A_1)=\ell+1$. By Proposition
	\ref{n6.16} we may find $\delta_1\subseteq \partialfn$ and
	$\delta_2\subseteq {}^{\underline{[N,M)}}2$ such that $\Delta(A_1)\preceq
	\delta_1$, $\Delta(A_2)\preceq\delta_2$, elements of $\delta_1$ have
	pairwise disjoint domains of size $k$ and elements of $\delta_2$ have
	pairwise disjoint domains of size $\ell$. Then 
	\begin{enumerate}
		\item[(a)] $D(\delta_i)\subseteq D(\Delta(A_i))=A_i$ for $i=1,2$, 
		\item[(b)] $D_M(\delta_1\cup\delta_2)=\big\{f\cup g: f\in D_N(\delta_1)\land
		g\in D_{[N,M)}(\delta_2)\big\}\subseteq A$,
		\item[(c)] $\Delta\big(D(\delta_1\cup\delta_2)\big)=\delta_0\cup\delta_2$
		(because of the disjoint domains).     
	\end{enumerate}
	Hence,
	\[\begin{array}{r}
	\norm_4^M(A)\geq \norm_4^M\big(D(\delta_1\cup\delta_2)\big)=
	\HN(\delta_1\cup\delta_2)\\
	= \min\big\{\HN(\delta_1),\HN(\delta_2)\big\}=\min\{k+1,\ell+1\}
	\end{array}\] 
	(remember Theorem \ref{thm6.30}).     
\end{proof}

\begin{proposition}
	\label{13A}
	Let $\delta\subseteq \partialfn$ be such that $\hn(\delta)>1$ and let
	$Z\subseteq N$. Then  
	\[\hn\big(L(\delta,Z)\big)\geq \frac{1}{2}\hn(\delta)\quad \mbox{ and }\quad 
	\hn\big(R(\delta,Z)\big)\geq \frac{1}{2}\hn(\delta).\]   
\end{proposition}

\begin{proof}
	Let $\hn(\delta)=k+1$. Suppose $\delta'\subseteq L(\delta,Z)$ and for
	$\sigma\in \delta'$ let $\varphi(\sigma)\in \delta$ be such that
	$\sigma=\varphi(\sigma)\restriction Z$ and $|\dom(\sigma)|\geq \frac{1}{2}
	|\dom\big(\varphi(\sigma)\big)$. We may find $\delta''\subseteq 
	\delta'$ such that the elements of $\{\varphi(\sigma):\sigma\in \delta''\}$
	have disjoint domains and 
	\[\Big|\bigcup_{\sigma\in\delta''} \dom\big(\varphi(\sigma)\big)\Big| \geq
	k\cdot |\delta'|.\] 
	Since $\dom(\sigma)=\dom\big(\varphi(\sigma)\big) \cap Z$ and
	$|\dom(\sigma)| \geq \frac{1}{2}|\dom\big(\varphi(\sigma)\big)$ we see that
	the elements of $\delta''$ have pairwise disjoint domains and 
	\[\Big|\bigcup_{\sigma\in \delta''} \dom(\sigma)\Big| \geq \frac{1}{2}
	k\cdot |\delta'|.\]
	Now we easily conclude that $\hn\big(L(\delta,Z)\big)\geq \lfloor\frac{1}{2}k
	\rfloor +1\geq \frac{1}{2}\hn(\delta)$. 
\end{proof}

\begin{proposition}
	\label{13B}
	Let $Z\subseteq N$ and suppose $A\subseteq {}^N2$ satisfies
	$\norm^N_4(A)>1$. Then there are sets $A_L\subseteq {}^Z2$ and
	$A_R\subseteq {}^{N\setminus Z} 2$ such that  
	\begin{enumerate}
		\item $\big\{f\in {}^N2: f\restriction Z\in A_L\land f\restriction
		(N\setminus Z)\in A_R\big\} \subseteq A$,
		\item $\norm_4^Z(A_L)\geq\frac{1}{2}\norm_4^N(A)$ and
		$\norm_4^{N\setminus Z}(A_R)\geq\frac{1}{2}\norm_4^N(A)$.   
	\end{enumerate}
\end{proposition}

\begin{proof}
	Let $k=\HN\big(\Delta_N(A)\big)-1$. By Proposition \ref{n6.16} we 
	may choose $\delta\subseteq \partialfn$ such that $\Delta_N(A)\preceq\delta$
	and elements of $\delta$ have disjoint domains of size $k$. Let 
	\[A_L=D_Z\big(L(\delta,Z)\big)\quad\mbox{ and }\quad A_R= D_{N\setminus Z} 
	\big(R(\delta,Z)\big).\] 
	Since elements of $L(\delta,Z)$ ($R(\delta,Z)$, respectively) have disjoint
	domains we see that $\Delta_Z\big(D_Z(L(\delta,Z))\big)=L(\delta,Z)$
	($\Delta_{N\setminus Z}\big(D_{N\setminus Z}(R(\delta,Z))\big)=R(\delta,Z)$,
	respectively). By Proposition \ref{13A} we get $\norm^Z_4(A_L)\geq
	\frac{1}{2}\hn(\delta)= \norm^N_4(A)$ and  $\norm^{N\setminus Z}_4(A_R)\geq 
	\frac{1}{2}\hn(\delta)= \norm^N_4(A)$. 
\end{proof}

\begin{theorem}
	For any $\delta \subseteq \partialfn$ and $A \subseteq N$, if $R(\delta, A)
	= \emptyset$ then $\HN(L(\delta, A)) \geq \HN(\delta) - \frac{N}{2}$. 
\end{theorem}

\begin{proof}
	Let $\delta \subseteq \partialfn$ and $A \subseteq N$ be arbitrary sets such
	that $R(\delta, A)  = \emptyset$. Suppose we have another set $\delta'$ such
	that $\delta \preceq \delta'$ and $\HN(\delta) = \hn(\delta')$. If
	$|\delta'| \geq 2$, then $\hn(\delta') \leq \frac{N}{2}$ and the proof is
	complete. 
	
	Otherwise, we  consider the set $\gamma = \{ \sigma\restriction (N\bs A) : 
	\sigma \in \delta \}$. Let $\rho \in \partialfn$ be a function such that
	$\gamma \preceq \{ \rho \}$ and $|\rho|$ is maximal. As we have $\gamma
	\preceq \{ \rho \}$, we also have $\delta \preceq \{\rho\}$. Notice that
	$|\rho| \leq \frac{N}{2}$. 
	
	Let $\lambda \in \partialfn$ be such that $L(\delta, A) \preceq \{\lambda\}$
	and $|\lambda|$ is maximal, and again we have $\delta \preceq \{ \lambda
	\}$. Together, we get $\delta \preceq \{ \lambda \cup \rho \}$, giving
	us 
	$$\HN(\delta) \geq \hn(\{\lambda\cup\rho\}) = |\lambda| + |\rho| + 1.$$ 
	To show this must be equality, we assume there is a $\theta \in \partialfn$
	such that $\delta \preceq \{\theta\}$ and $|\theta| > | \lambda | + | \rho
	|$. Checking $\theta \restriction A$, we get $L(\delta, A) \preceq \{
	\theta\restriction A \}$. If $|\theta \restriction A | > |\lambda|$, then we
	contradict $|\lambda|$ being maximal, so we have $|\theta \restriction A|
	\leq |\lambda|$. Checking $\theta \restriction (N\bs A)$, we must either
	contradict $|\rho|$ being maximal or we get $|\theta \restriction (N\bs A)|
	\leq |\rho|$. Finally, as $\theta = \theta \restriction A \cup \theta
	\restriction (N\bs A)$, we must have $|\theta| \leq |\lambda \cup \rho|$,
	which contradicts our assumption for $\theta$. Therefore we have
	$\HN(\delta) = |\lambda| + 1 +  |\rho| \leq \HN(L(\delta, A)) +
	\frac{N}{2}$. 
\end{proof}

\begin{remark}
	This is the best limit we can have. Consider when $A \subseteq N$ is such
	that $|A| = \frac{N}{2}$. Let $\delta = \{ f \}$, where $f$ is any element
	of ${^N2}$. We get $R(\delta, A) = \emptyset$, $L(\delta, A) = \{ f
	\restriction A \}$, $HN(\delta) = N+1$, and $\HN(L(\delta, A)) = \frac{N}{2}
	+ 1$. 
	
	This isn't the only case, we can also consider a set $A \subseteq N$ such
	that $|A| = \frac{N}{2}$. Let $\rho \in \partialfn$ be such that $A
	\subseteq \dom(\rho)$. We define $\delta = \{ \sigma \in \partialfn : \rho
	\subseteq \sigma \}$. Clearly $\HN(\delta) = |\rho| + 1$, but we also get
	$R(\delta, A) = \emptyset$ and $\HN(L(\delta, A)) = \HN(\delta) -
	\frac{N}{2}$. 
\end{remark}

\subsection{Relationship with the Standard Norm}

\begin{theorem}
	If we have $\norm_4(A) = k+1>1$, then 
	\[|A| \geq 2^N -
	\sum\limits_{j=1}^{\lfloor\frac{N}{k}\rfloor}(-1^{j-1}{\lfloor
		\frac{N}{k}\rfloor \choose j}2^{N-jk}).\]  
\end{theorem}

\begin{proof}
	By definition, $\norm_4(A) = k+1$ implies that there is some set $\delta$
	such that $\hn(\delta) = k+1$ and $\Delta(A) \preceq \delta$. Let $\delta_0$
	satisfy this. We also have that $\hn(\delta_0) = k+1$ implies the existence
	of a $k$-selector for $\delta_0$. Let $G$ be a $k$-selector for
	$\delta_0$. Define the set $\delta' = \{ \sigma \restriction G(\sigma)
	:\sigma \in \delta_0 \}$. Notice that $\delta_0 \preceq \delta'$,
	$\hn(\delta') = k+1$, and $\Delta(D(\delta')) = \delta'$. Therefore, we have
	$\Delta(A) \preceq \Delta(D(\delta'))$, and consequently $D(\delta')
	\subseteq A$. We then have $|A| \geq |D(\delta')| = 2^N - |
	\bigcup\limits_{\sigma \in \delta'} [\sigma] |$. As we are trying to create
	a lower bound, we will assume $|\delta'| = \lfloor \frac{N}{k}
	\rfloor$. Furthermore, by the general inclusion-exclusion principle,  
	\[| \bigcup\limits_{\sigma \in \delta'} [\sigma] | =
	\sum\limits_{\sigma\in\delta'} |[\sigma]| - \sum\limits_{\sigma_1, \sigma_2}
	| [\sigma_1]\cap[\sigma_2]| + \ldots + (-1)^{\lfloor \frac{N}{k}
		\rfloor}|[\sigma_1] \cap [\sigma_2]\cap\ldots\cap[\sigma_{\lfloor
		\frac{N}{k} \rfloor}]|.\] 
	We already know that the functions in $\delta'$ are pairwise disjoint, so we
	can simplify to 
	\[ | \bigcup\limits_{\sigma \in \delta'} [\sigma] | =
	\sum\limits_{\sigma\in\delta'} |[\sigma]| - \sum\limits_{\sigma_1, \sigma_2}
	| [\sigma_1\cup\sigma_2]| + \ldots + (-1)^{\lfloor \frac{N}{k}
		\rfloor}|[\sigma_1\cup\ldots\cup\sigma_{\lfloor \frac{N}{k} \rfloor}]|. \] 
	We also know that $|[\sigma]| = 2^{N-|\sigma|}$, which allows us to further
	simplify to  
	\[ | \bigcup\limits_{\sigma \in \delta'} [\sigma] | =
	\sum\limits_{\sigma\in\delta'} 2^{N-k} - \sum\limits_{\sigma_1, \sigma_2}
	2^{N-2k} + \ldots + (-1)^{\lfloor \frac{N}{k}
		\rfloor}2^{N-\lfloor\frac{N}{k}\rfloor k}. \] 
	Finally, recognizing that for the $j^{th}$ sum, we are summing over all ways
	to choose $j$ functions from $\delta'$, we get  
	\[ | \bigcup\limits_{\sigma \in \delta'} [\sigma] | =
	\sum\limits_{j=1}^{\lfloor\frac{N}{k}\rfloor}(-1)^{j-1}{\lfloor
		\frac{N}{k}\rfloor \choose j}2^{N-jk}. \] 
	As we are subtracting this off of $2^N$, we get $|A|\geq 2^N -
	\sum\limits_{j=1}^{\lfloor\frac{N}{k}\rfloor}(-1^{j-1}{\lfloor
		\frac{N}{k}\rfloor \choose j}2^{N-jk})$. 
\end{proof}

\subsection{Relationship with the Subset Norm}
For this subsection, we will assume $G = 2^N$, and $n \leq N$. 

\begin{definition}
	We define the function $P: \partialfn \rightarrow \pow(N)$ as follows  
	\[ P(\sigma) = \{ a \in \dom(\sigma) : \sigma(a) = 1 \} \quad\mbox{ for
	}\sigma\in\partialfn.\] 
\end{definition}

\begin{observation}
	$P\restriction {}^N 2$ is a bijection from ${^N2}$ onto $\pow(N)$.
\end{observation}

\begin{proposition}
	Suppose we have a set $B \subseteq X_n^G$ and $A=P^{-1}[B]\cap {}^N2$. If we
	know that $\norm_2(B) = k$, then $\norm_4(A) \leq k+1$. 
\end{proposition}

\begin{proof}
	Let $B \subseteq X_n^G$ and $A=P^{-1}[B]\cap {}^N2$ and suppose that
	$\norm_2(B)= k$. Take $x\subseteq G$ such that
	$(\forall b \in B) x \nsubseteq b$ and $|x| = k$. Consider the partial
	function $\sigma:x\longrightarrow 2$ with constant value $1$. We have
	$|\sigma| = k$, and also that $[\sigma]\cap A = \emptyset$. We can conclude 
	that $\sigma \in \Delta(A)$, so for any set $\delta$ such that $\Delta(A)
	\preceq \delta$, there is a $\rho \subseteq \sigma$ where $\rho \in
	\delta$. This gives us $\hn(\delta) \leq |\rho|+1 \leq |\sigma|+1 =
	k+1$. Finally, as this is for every $\delta$, we have $\norm_4(A) =
	\HN(\Delta(A)) \leq k+1$. 
\end{proof}

\begin{definition}
	For any set $A \subseteq {^N2}$ we define 
	\[P^*(A) = \{ P(f) : f \in A \land |P(f)| = \frac{G}{2^n} = H \} .\]
\end{definition}

\begin{proposition}
	Suppose we have a set $A\subseteq {^N2}$, and for some $\sigma\in\partialfn$
	we know $\sigma \in \Delta(A)$, then $(\forall p \in P^*(A))P(\sigma)
	\nsubseteq p$. 
\end{proposition}

\subsection{Relationship with the Graph Coloring Norm}
The result of this section suggests that there should not be a
strong connection between the Graph Coloring Norm, and Hall's Marriage
theorem. 

\begin{definition}
	We define the function $P^+:\pow({^N2}) \rightarrow \pow(N)$ as follows
	\[ P^+(A) = \{ P(f) : |P(f)| \geq 2 \land f\in A \} \quad\mbox{ for }
	A\subseteq {}^N 2.\] 
\end{definition}

\begin{lemma}
	For any $\sigma \in \partialfn$, any edge of $\dom(\sigma)$ that is not
	equal to $P(\sigma)$ is in $P^+(D(\{ \sigma \}))$. 
\end{lemma}

\begin{proof}
	Let $\{ a,b \} \subseteq \dom(\sigma)$ be an edge such that $P(\sigma) \neq
	\{ a,b \}$. We are going to argue that $\{a,b\}=P(f)$ for some $f\in
	{}^N2\setminus [\sigma]$. For this we consider two cases. 
	
	In the first case $\sigma(a) = 0 \lor\sigma(b) = 0$. Without loss of
	generality, we assume $\sigma(a) = 0$. As $\sigma(a) = 0$ and $f =
	P^{-1}(\{a,b\})$ is such that $f(a) = 1$, we get $f \in {}^N2\setminus
	[\sigma]$.  
	
	In the second case, we have $\sigma(a) = 1 \land \sigma(b) = 1$. However, as
	$\{ a, b \} \neq P(\sigma)$ it must be that there is a $c$ such that $c \neq
	a$, $c \neq b$, and $\sigma(c) = 1$. Thus, as $f = P^{-1}(\{a,b\})$ is such
	that $f(c) = 0$, we get $f \in {}^N 2\setminus [\sigma]$.  
\end{proof}

\begin{theorem}
	For any set $A \subseteq {^N2}$, if there is a $\sigma \in \partialfn$ such
	that $\Delta(A) \preceq \{\sigma \}$, then $P^+(A)$ requires at least
	$|\sigma| -1$ sets to be split. 
\end{theorem}

\begin{proof}
	Suppose $A \subseteq {^N2}$ and $\Delta(A) \preceq \{ \sigma \}$ for some
	$\sigma \in \partialfn$. For convenience, we will say $B =
	D(\{\sigma\})$. Notice that $\Delta(B) = \{ \sigma \}$. As $\Delta(A)
	\preceq \{ \sigma \}$, we get $B \subseteq A$, which gives us $P^+(B)
	\subseteq P^+(A)$. As a result, we get that if $P^+(B)$ cannot be split by
	$V_0, \ldots, V_{n-1}$, then $P^+(A)$ cannot be split by $V_0, \ldots,
	V_{n-1}$. But as every edge of $\dom(\sigma)$ that is not equal to
	$P(\sigma)$ is in $B$, it must be that $B$ requires at least $|\dom(\sigma)|
	- 1$ sets to be split, and therefore $A$ requires at least $|\sigma| - 1$
	sets to be split. 
\end{proof}

\begin{corollary}
	If $A \subseteq {^N2}$ is such that $\norm_4(A) = k+2 \geq \frac{N}{2} + 1$,
	then $P^+(A)$ requires at least $k$ sets to be split. 
\end{corollary}

\begin{observation}
	There is a set $A \subseteq {^N2}$ such that  $\HN(A) =
	\lfloor\frac{N}{2}\rfloor +1$ and $P^+(A)$ can be split by two sets. 
\end{observation}

\begin{proof}
	Consider any set $\delta\subseteq \partialfn$ where 
	\begin{enumerate}
		\item  $|\delta| = 2$,
		\item if $\delta = \{ \sigma_1, \sigma_2\}$, then $\dom(\sigma_1) \cap
		\dom(\sigma_2) = \emptyset$, 
		\item and $(\forall \sigma \in \delta) (\forall a \in \dom(\sigma))
		\sigma(a) = 0$, 
		\item $|\sigma_1|=|\sigma_2|=\lfloor\frac{N}{2}\rfloor$, so $\hn(\delta) = 
		\lfloor \frac{N}{2}\rfloor+1$.  
	\end{enumerate}
	Let $A = D(\delta)$, we have $\Delta(A) = \delta$. Finally if $\sigma \in
	\delta$, we have $P^+(A)$ can be split by $\dom(\sigma), N\bs \dom(\sigma)$.  
	
\end{proof}

\begin{observation}
	Let $A = \{ f \in {^N2} :  |\{ a \in N : f(a) = 1 \}| = 2 \}$. Clearly we
	have $P^+(A) = \{ e \subseteq N : |e| = 2 \}$. We also know that $P^+(A)$
	cannot be split by $N-1$ sets. Consider any $\sigma \in\partialfn$ where
	$|\sigma| = 3$ and $(\forall a \in \dom(\sigma)) \sigma(a) = 1$, we have
	$[\sigma] \cap A = \emptyset$. Consequently, $\sigma \in
	\Delta(A)$. Furthermore if $\Delta(A) \preceq \delta$, then there is a $\rho
	\subseteq \sigma$ such that $\rho \in \delta$. This gives us $\norm_4(A)
	\leq 4$. 
\end{observation}

\begin{proposition}
	If $A \subseteq {^N2}$ and $p \subseteq N$ are such that
	for every $p\subseteq q \subseteq N$ we have $q \notin P^+(A)$, then
	$\norm_4(A) \leq |p|+1$. 
\end{proposition}

\begin{proof}
	Suppose $A \subseteq {^N2}$ and $p\subseteq N$ satisfy 
	\[(\forall q \subseteq N)( p \subseteq q \implies q \notin P^+(A)).\]
	If $p=\emptyset$, then $P^+(A)=\emptyset$ and hence $\Delta(A)$ consists of
	functions with domains of size $\leq 2$ and values $1$. Hence  $\norm_4(A)=1$.  
	
	\noindent If $p=\{x\}$, then for $f\in A$ we have that either $f(x)=0$ or
	else $f(y)=0$ for all $y\neq x$. Now easily $\norm_4(A)\leq 2$.  
	
	Assume now that $|p|\geq 2$. Let $\sigma\in \partialfn$ be such that
	$\dom(\sigma) = p$ and $\sigma(a) = 1$ for all $a \in p$. We have $[\sigma]
	\cap A = \emptyset$, therefore there is a $\rho \subseteq \sigma$ such that
	$\rho \in \Delta(A)$. This gives us $\norm_4(A) \leq |\rho| + 1 \leq
	|\sigma| + 1 = |p| + 1$.  
\end{proof}

\begin{corollary}
	For any $A \subseteq {^N2}$, if for some $n \leq N$ we have $(\forall p \in
	P^+(A)) |p| \leq n$, then $\norm_4(A) \leq n + 1$. 
\end{corollary}

Thus we see that there are sets where the Graph Coloring norm has a large
value, and the Hall norm has a small value, and vice-versa, that is, there
is not a strong relationship between the two norms.


\begin{thebibliography}{1}

\bibitem{Baju95}
Tomek Bartoszy\'nski and Haim Judah.
\newblock {\em {Set Theory: On the Structure of the Real Line}}.
\newblock A K Peters, Wellesley, Massachusetts, 1995.

\bibitem{FrSh:406}
David H. Fremlin and Saharon Shelah. 
\newblock {Pointwise compact and stable sets of measurable functions}. 
\newblock {\em Journal of Symbolic Logic}, 58:435--455, 1993.

\bibitem{RoSh:628}
Andrzej Roslanowski and Saharon Shelah.
\newblock {Norms on possibilities II: More ccc ideals on
  $2^{\textstyle\omega}$}.
\newblock {\em {Journal of Applied Analysis}}, 3:103--127, 1997.
\newblock arxiv:math.LO/9703222.

\bibitem{RoSh:470}
Andrzej Roslanowski and Saharon Shelah.
\newblock {Norms on possibilities I: forcing with trees and creatures}.
\newblock {\em {Memoirs of the American Mathematical Society}}, 141(671):xii +
  167, 1999.
\newblock arxiv:math.LO/9807172.

\bibitem{RoSh:672}
Andrzej Roslanowski and Saharon Shelah.
\newblock {Sweet {\&} Sour and other flavours of ccc forcing notions}.
\newblock {\em Archive for Mathematical Logic}, 43:583--663, 2004.
\newblock arxiv:math.LO/9909115.

\bibitem{Sh:207}
Saharon Shelah. 
\newblock {On cardinal invariants of the continuum}.
\newblock In {\em Axiomatic set theory (Boulder, Colo., 1983)}, volume 31 of
{\em Contemp. Mathematics},
\newblock pages 183--207. Amer. Math. Soc., Providence, RI, 1984.

\bibitem{Sh:f}
Saharon Shelah.
\newblock {\em {Proper and improper forcing}}.
\newblock {Perspectives in Mathematical Logic}. {Springer}, 1998.

\end{thebibliography}
\end{document}